\documentclass[12pt]{article}
\usepackage{amsmath,amssymb,amsfonts,amsthm}
\usepackage[dvips]{graphicx}
\usepackage{natbib}
\usepackage{array}

\makeatletter
\@addtoreset{equation}{section}
\@addtoreset{figure}{section}
\@addtoreset{table}{section}

\makeatother

\newtheorem{rem}{Remark}[section]

\newtheorem{prop}{Proposition}[section]
\newcommand\R{\mathbb{R}}
\renewcommand\S{\mathbb{S}}
\newcommand{\Cov}{\mathrm{Cov}}
\newcommand{\vol}{\mathrm{Vol}}
\newcommand{\argmin}{\mathop{\mathrm{argmin}}}
\begin{document}
\title{
Likelihood ratio tests for positivity in polynomial regressions
}
\setcounter{footnote}{1}
\author{
Naohiro Kato
\!\thanks{
The Institute of Statistical Mathematics,
10-3 Midoricho, Tachikawa, Tokyo 190-8562, Japan.
Email:\,\texttt{nkato@ism.ac.jp}
}
\and
Satoshi Kuriki
\!\thanks{
The Institute of Statistical Mathematics,
10-3 Midoricho, Tachikawa, Tokyo 190-8562, Japan.
Email:\,\texttt{kuriki@ism.ac.jp}
}}
\date{}
\maketitle
\begin{abstract}
A polynomial that is nonnegative over a given interval
is called a positive polynomial.
The set of such positive polynomials forms a closed convex cone $K$.
In this paper, we consider the likelihood ratio test for the hypothesis
of positivity that the estimand polynomial regression curve is a positive polynomial.
By considering hierarchical hypotheses including the hypothesis of positivity,
we define nested likelihood ratio tests, and derive their null distributions
as mixtures of chi-square distributions by using the volume-of-tubes method.
The mixing probabilities are obtained by utilizing the parameterizations
for the cone $K$ and its dual provided in the framework of
Tchebycheff systems for polynomials of degree at most 4.
For polynomials of degree greater than 4, the upper and lower bounds
for the null distributions are provided.
Moreover, we propose associated simultaneous confidence bounds for
polynomial regression curves.
Regarding computation,
we demonstrate that symmetric cone programming is useful
to obtain the test statistics.
As an illustrative example,
we conduct data analysis on growth curves of two groups.
We examine the hypothesis that
the growth rate (the derivative of growth curve) of one group is
always higher than the other.

\smallskip\noindent
\textit{Key words:}\,%
chi-bar square distribution,
cone of positive polynomials,
moment cone,
symmetric cone programming,
Tchebycheff system,
volume-of-tubes method.
\end{abstract}


\section{Introduction}
\label{sec:intro}

Consider the polynomial regression model of degree $n$ ($\ge 1$):
\begin{equation}
 y_h = f(t_h;c) + \varepsilon_h, \quad
 f(t;c) = c^\top \psi(t) = \sum_{i=0}^n c_i t^i, \quad t\in T,
\label{regression}
\end{equation}
$h=1,\ldots,N$,
where $\psi(t)\,(=\psi_n(t)) =(1,t,\ldots,t^n)^\top$ and
$c=(c_0,c_1,\ldots,c_n)^\top$ are column vectors in $\R^{n+1}$.
The errors $\varepsilon_h$ are independently distributed
according to the normal distribution
$N(0,\sigma^2)$ with mean 0 and variance $\sigma^2$.
$T\subseteq\R$ is the region of the explanatory variable $t$
where the model (\ref{regression}) is defined.
Typically, $T$ is a bounded interval in $\R$.
We assume that sufficient statistics of the model (\ref{regression}),
that is, the ordinary least squares estimator $\widehat c$ of $c$, and
when $\sigma^2$ is unknown,
the unbiased variance estimator $\widehat\sigma^2$ of $\sigma^2$,
distributed independently of $\widehat c$,
are available.

In this paper, we deal with the hypothesis of positivity,
or of superiority:
\begin{equation}
\label{positivity}
 f(t;c) \ge 0 \quad\mbox{for all $t\in T$}.
\end{equation}
To state its statistical meaning,
it is natural to consider a two-sample problem.
Let $f(t;c_{(j)})=c_{(j)}^\top \psi(t)$ ($j=0,1$)
be the polynomial regression curves of two groups.
The hypothesis that the polynomial curve of group $1$ is always
bounded below by, or superior to,
group $0$ is expressed as $f(t;c_{(1)}) \ge f(t;c_{(0)})$ for all $t\in T$.
Taking the difference, we see that (\ref{positivity}) represents
the hypothesis of superiority.
It is also possible to model the difference of two profiles (mean vectors)
by a polynomial without modeling the profile of each group
 (see Section \ref{subsec:example}).
This notion of superiority is particularly important
in statistical tests for assessing new drugs (\cite{Liu-etal09}).

The set of coefficients $c$ satisfying (\ref{positivity}) forms
a closed convex cone:
\begin{equation}
\label{K}
 K \,(= K_n)
 = \bigl\{ c \in \R^{n+1} \mid c^\top \psi(t) \ge 0,\ \forall t\in T \bigr\}.
\end{equation}
This is referred to as the cone of positive polynomials (\cite{Barvinok02}).
We use $K_n$ instead of $K$ when we emphasize that $K$ is defined in $\R^{n+1}$.
$K$ is closed, since
$K = \bigcap_{t\in T} \bigl\{ c \mid c^\top \psi(t) \ge 0\}$
is the intersection of closed sets.
The hypothesis (\ref{positivity}) is rewritten as $c\in K$.
Including this hypothesis, we consider the following hierarchical hypotheses:
\begin{equation}
\label{H012}
H_0: c=0, \quad H_1: c\in K, \quad\mbox{and}\ \ %
H_2: c\in \R^{n+1}\ (\mbox{$c$ is unrestricted}).
\end{equation}
We then formalize the test for positivity
as the likelihood ratio test (LRT) for testing $H_1$ against $H_2$.
In addition, we define an LRT for testing $H_0$ against $H_1$.
In the context of the two-sample problem,
this is the test for the equality of two regression curves
against the hypothesis of superiority.
As we see later, it is mathematically convenient to treat the two LRTs
at a time.

The theory of LRTs for convex cone hypotheses
has been developed under the name of order restricted inference 
 (\cite{Robertson-etal88}).
A general theorem states that the null distribution of LRT statistics is
a finite mixture of chi-square distributions (\cite{Shapiro88}).
When the cone has piecewise smooth boundaries,
\cite{Takemura-Kuriki97,Takemura-Kuriki02}
proved that the weights (mixing probabilities) are expressed
in terms of curvature measures on boundaries.
These results arise out of a geometric approach referred to as
the volume-of-tubes method.
Using this method, \cite{Kuriki-Takemura00} gave the weights
associated with the cone of nonnegative definite matrices.
However, the weights of few cones are obtained explicitly.

The main result of the present paper is the derivation of
the weights associated with the cone of positive polynomials $K$,
that is, the null distribution of the LRT for positivity.
By applying the representation (parameterization) theorem for
the positive polynomial cone and its dual cone
developed in the framework of Tchebycheff systems
(\cite{Karlin-Studden66}), we evaluate the weights
of the two highest degrees ($w_{n+1},w_n$) and the two lowest degrees
($w_0,w_n$).  In terms of these weights, the null distributions of the LRTs
are expressed
when the degree $n$ of the polynomial regression is less than or equal to $4$.
When the degree $n$ is more than 4, upper and lower bounds for the null
distributions are provided.

The outline of the paper is as follows.
In Section \ref{sec:lrt}, we present the expressions of the LRT statistics
in both cases where the variance $\sigma^2$ is known and unknown.
As in most statistical tests, we can also propose
simultaneous confidence bands associated with the LRT for positivity.
In Section \ref{sec:null},
we first briefly summarize the volume-of-tubes method.
In order to apply this method, we need
the volumes of the cone $K$, its dual cone, and their boundaries.  
Modifying the representation theorems for the positive polynomials
in Tchebycheff systems, we obtain explicit formulas for the weights.
In Section \ref{sec:computation}, we discuss computation.
To construct our LRT statistics, we need the maximum likelihood estimate
 (MLE) $f(t;\widehat c_K)$, say, under the hypothesis of positivity.
The coefficient $\widehat c_K$ is calculated as the orthogonal projection
of $\widehat c$ onto the positive polynomial cone $K$.
We show that this calculation can be conducted by
symmetric cone programming, which is extensively studied
in the optimization community.
We also demonstrate an example of growth curve data analysis.

Throughout the paper, we treat only the polynomial regression.
However, a polynomial is just one example of Tchebycheff systems.
The approach developed here is applicable to other systems.
Another typical example is trigonometric regression
\[
 f(\theta;c) = c_0 + \sum_{i=1}^{n/2} \bigl\{ c_{2i-1} \sin(i \theta) + c_{2i} \cos(i\theta) \bigr\}, \quad \theta\in\Theta\subseteq [0,2\pi),
\]
and we can consider the testing problem for the positivity once more.
In this case, by changing a variable $t=\tan(\theta/2)$,
all results in the polynomial regression are translated
into the trigonometric regression.

\section{Likelihood ratio tests and confidence bands}
\label{sec:lrt}

\subsection{Likelihood ratio test statistics}
\label{subsec:lrt}

Throughout the paper, we need to deal with a metric linear space and its dual
space simultaneously.
We write the inner product and the norm as
\[
 \langle x,y\rangle_Q = x^\top Q y, \quad
 \Vert x\Vert_Q = \sqrt{\langle x,x\rangle_Q},
\]
where $Q$ is a positive definite matrix.
The orthogonal projection of $x$ onto the set $A$ with respect to the distance
$\Vert\,\Vert_Q$ is denoted by
\[
 \Pi_Q(x|A) = \argmin_{y\in A} \Vert x-y\Vert_Q.
\]
This is well defined when $A$ is a closed convex set.
The subscript $Q$ in $\langle\,,\,\rangle_Q$, $\Vert\,\Vert_Q$, and $\Pi_Q$
will be omitted when it does not cause any confusion.

In the regression model (\ref{regression}) with $\sigma^2$ known,
the least squares statistic $\widehat c$ is sufficient,
and we can restrict attention to inference based on $\widehat c$.
The distribution of $\widehat c$ is the $(n+1)$-dimensional normal distribution
$N_{n+1}(c,\Sigma)$ with mean vector $c$ and covariance matrix $\Sigma$,
where $\Sigma=\sigma^2 \Sigma_0$ with
$\Sigma_0=\bigl(\sum_{i=1}^N \psi(t_i) \psi(t_i)^\top\bigr)^{-1}$,
the inverse of the design matrix.
When $\sigma^2$ is unknown, the sufficient statistic is the pair
$(\widehat c,\widehat\sigma^2)$,
where $\widehat\sigma^2$ is the unbiased estimator of $\sigma^2$
calculated from the residuals, and
whose distribution is proportional to that of a chi-square random variable
with $\nu=N-n-1$ degrees of freedom.

Given the data $\widehat c$ distributed as the normal distribution
$N_{n+1}(c,\Sigma)$ with $\Sigma=\sigma^2\Sigma_0$ known,
the MLE of $c$ under the hypothesis of positivity
$H_1: c\in K$ is the orthogonal projection
$\widehat c_K$ of $\widehat c$ onto the cone $K$
under the metric $\langle\,,\,\rangle_{\Sigma^{-1}}$.
When $\sigma^2$ is unknown, the MLE is the orthogonal projection onto $K$
under the metric $\langle\,,\,\rangle_{\widehat\Sigma^{-1}}$,
$\widehat\Sigma=\widehat\sigma^2\Sigma_0$.
This MLE is the same as that with $\Sigma$ known,
because the orthogonal projection onto a cone is invariant
with respect to the scale change of metric
$\langle\,,\,\rangle_Q \mapsto \langle\,,\,\rangle_{kQ}$ ($k>0$).
The MLEs of $c$ under $H_0$ and $H_2$ are given as $0$ and
$\widehat c$, respectively.
Acknowledging these facts, we obtain the LRT statistics as follows.
\begin{prop}
\label{prop:lrt}
When the variance $\sigma^2$ is known,
the LRT statistics for $H_0$ against $H_1$, and for $H_1$ against $H_2$
are given by
\begin{equation}
\label{lambda}
 \lambda_{01} = \Vert \widehat c_K \Vert^2_{\Sigma^{-1}} \quad\mbox{and}\quad
 \lambda_{12} = \Vert \widehat c \Vert^2_{\Sigma^{-1}}
              - \Vert \widehat c_K \Vert^2_{\Sigma^{-1}},
\end{equation}
respectively, where $\widehat c_K=\Pi_{\Sigma^{-1}}(\widehat c|K)$.

When the variance $\sigma^2$ is unknown and
an independent and unbiased estimator
$\widehat\sigma^2$ of $\sigma^2$ with $\nu$ degrees of freedom is available,
the LRT statistics for $H_0$ against $H_1$, and for $H_1$ against $H_2$
are given by
\begin{equation}
\label{beta}
 \beta_{01} = \frac{\Vert \widehat c_K \Vert^2_{\widehat\Sigma^{-1}}}{\Vert \widehat c\Vert^2_{\widehat\Sigma^{-1}} + \nu}
 \quad\mbox{and}\quad
 \beta_{12} = \frac{\Vert \widehat c \Vert^2_{\widehat\Sigma^{-1}} - \Vert \widehat c_K \Vert^2_{\widehat\Sigma^{-1}}}{\Vert \widehat c \Vert^2_{\widehat\Sigma^{-1}} - \Vert \widehat c_K \Vert^2_{\widehat\Sigma^{-1}} + \nu},
\end{equation}
respectively, where
$\widehat\Sigma = \widehat\sigma^2 \Sigma_0$,
$\widehat c_K = \Pi_{\widehat\Sigma^{-1}}(\widehat c|K)$.

The null hypotheses are rejected when the LRT statistics are
sufficiently large.
\end{prop}

The hypothesis of positivity $H_1$ is a composite hypothesis.
To obtain the critical points for testing such a hypothesis,
we need to know the least favorable configuration.
The proof of the following proposition is essentially given in Section 2.3 of
\cite{Robertson-etal88}.
\begin{prop}
\label{prop:lfc}
In both cases where $\sigma^2$ is known or unknown,
the least favorable configurations of the LRTs
for testing $H_1$ (the hypothesis of positivity) against
$H_2$ (the no-restriction hypothesis) are given by the case
where $H_0$ holds, that is, $c=0$.
\end{prop}
\begin{proof} 
In the case where $\sigma^2$ is known, the acceptance region is of the form
\[
 A=\Bigl\{ x\in\R^{n+1} \mid \min_{y\in K} \Vert x-y \Vert < d \Bigr\}.
\]
We first prove the monotonicity of the set $A$:
\[
 A-c = \{ x-c \mid x\in A\} \supseteq A \quad \mbox{for any $c\in K$}.
\]
This is because, for $c\in K$,
\begin{align*}
A-c
&= \Bigl\{ x-c \mid \min_{y\in K} \Vert x-y \Vert < d \Bigr\} \\
&= \Bigl\{ x \mid \min_{y\in K} \Vert x+c-y \Vert < d \Bigr\} \\
&= \Bigl\{ x \mid \min_{y\in K-c} \Vert x-y \Vert < d \Bigr\}
 \supseteq \Bigl\{x \mid \min_{y\in K} \Vert x-y \Vert < d \Bigr\} =A.
\end{align*}
The last inclusion follows from $K\subseteq K-c$, because $K$ is a convex cone.
Therefore, for $X \sim N_{n+1}(c,\Sigma)$,
\begin{align*}
P(X \in A \mid c)
&= P(X+c \in A \mid c=0) \\
&= P(X \in A-c \mid c=0) \ge P(X \in A \mid c=0),
\end{align*}
and $\inf_{H_1} P(X \in A \mid c) = P_{H_0}(X \in A)$ follows.

In the case where $\sigma^2$ is unknown, the LRT statistic $\beta_{12}$
in (\ref{beta}) is rewritten as
\[
\beta_{12} = \frac{\Vert \widehat c \Vert^2_{\Sigma^{-1}} - \Vert \widehat c_K \Vert^2_{\Sigma^{-1}}}{\Vert \widehat c \Vert^2_{\Sigma^{-1}} - \Vert \widehat c_K \Vert^2_{\Sigma^{-1}} + \nu \widehat\sigma^2/\sigma^2}
= \frac{\lambda_{12}}{\lambda_{12} + \nu \widehat\sigma^2/\sigma^2},
\]
which is monotone in $\lambda_{12}$ in (\ref{lambda}).
The monotonicity of the acceptance region can be proved similarly.  
\end{proof}

\subsection{Simultaneous confidence bounds}
\label{subsec:confidence}

In regression analysis, simultaneous confidence bounds
for the estimated regression curve are often provided
to assess the reliability of the estimated regression curve.
The construction of confidence bands is still an active research topic
because of its practical importance (\cite{Liu10}).
In this subsection, we propose simultaneous confidence bands
that are naturally linked to our proposed LRTs.

In general, when we want to construct simultaneous confidence bands
for the regression curve $\{ f(t;c) \mid t\in T \}$, we need to bound
$|f(t;c)-f(t;\widehat c)| = |(c-\widehat c)^\top \psi(t)|$
above by a pivotal statistic whose distribution
is independent of the true parameters.
The most standard tool to obtain the upper bound is
the Cauchy-Schwarz inequality.
However, in this inequality, strict equality is attained
when and only when (the closure of) the set of undirected rays spanned by
the explanatory variable vectors
$\{ \alpha \psi(t) \mid t\in T,\ \alpha\in \R \}$ forms the whole space.
In our polynomial regression model (\ref{regression}),
this becomes the whole space $\R^{n+1}$ only when $n=1$ and $T=\R$
(\cite{Working-Hotelling29}). 
The cases where $n\ge 2$ or $T$ is a proper subset of $\R$ ($T\subsetneq\R$)
are not easy problems and have been solved in limited cases
 (e.g., \cite{Uusipaikka83}, \cite{Wynn-Bloomfield71}).

In our proposal, we relax the set of estimands from the regression curve
itself.
Let $\mu(dt)$ be a nonnegative measure on $T\subseteq \R$,
and write $\mu[\psi]=\int_T \psi(t) \, \mu(dt)$.
Then, $\mu[\psi]\in K^*$, where
\begin{equation}
\label{Ks}
 K^* \,(= K_{n}^*)
 = \overline{\bigl\{ \mu[\psi] \mid \mu(dt)\ge 0 \bigr\}}
\end{equation}
is the closure of the conic hull of the trajectory $\{\psi(t) \mid t\in T\}$.
This cone is the dual cone of the positive polynomial cone $K$ in (\ref{K}),
and is referred to as the moment cone (\cite{Barvinok02}).
We construct confidence bands on the basis of the inequality
\begin{align}
\mu[f_{\widehat c}] - \mu[f_c]
 = \int_T (\widehat c-c)^\top \psi(t) \,\mu(dt)
& = \langle \Sigma^{-1}(\widehat c-c),\mu[\psi]\rangle_{\Sigma}  \nonumber \\
& \le \bigl\Vert \mu[\psi] \bigr\Vert_{\Sigma} \cdot
  \bigl\Vert \Pi_{\Sigma}(\Sigma^{-1}(\widehat c-c)|K^*) \bigr\Vert_{\Sigma},
\label{inequality}
\end{align}
where $\mu[f_c]=\int f(t;c)\,\mu(dt)$.
The equality in (\ref{inequality}) holds for some $\mu$ if and only if
$\widehat c-c \notin -K\setminus\{0\}$, where
$K$ is the positive polynomial cone in (\ref{K}).

The statistic
$\bigl\Vert \Pi_{\Sigma}(\Sigma^{-1}(\widehat c-c)|K^*) \bigr\Vert_{\Sigma}$
in (\ref{inequality}) is distributed independently of the true parameter $c$.
Moreover, its square is rewritten as
\begin{align*}
\bigl\Vert \Sigma^{-1}(\widehat c-c) \bigr\Vert_{\Sigma}^2
- \min_{x\in K^*} \bigl\Vert \Sigma^{-1}(\widehat c-c) -x \bigr\Vert_{\Sigma}^2
&= \bigl\Vert \widehat c-c \bigr\Vert_{\Sigma^{-1}}^2
- \min_{y\in \Sigma K^*} \bigl\Vert \widehat c-c -y \bigr\Vert_{\Sigma^{-1}}^2
 \\
&= \bigl\Vert\Pi_{\Sigma^{-1}}(\widehat c-c|\Sigma K^*)\bigr\Vert_{\Sigma^{-1}}^2 \\
&= \Vert \widehat c-c \Vert_{\Sigma^{-1}}^2 -
 \bigl\Vert \Pi_{\Sigma^{-1}}(\widehat c-c|K) \bigr\Vert_{\Sigma^{-1}}^2,
\end{align*}
which has the same distribution as $\lambda_{12}$ in (\ref{lambda})
under $H_0:c=0$.
Using the upper $\alpha$ percentile $\lambda_{12,\alpha}$ of the distribution
of $\lambda_{12}$ under $H_0$,
we obtain the $1-\alpha$ simultaneous confidence bands as follows.
\begin{prop}
\label{prop:confidence}
The statement below holds with probability $1-\alpha$:
\[
 \mu[f_c] \in
 \bigl( \mu[f_{\widehat c}]
 - \sqrt{\lambda_{12,\alpha}}\,\Vert\mu[\psi]\Vert_{\Sigma},\infty \bigr)
 \quad\mbox{for all nonnegative measures $\mu$ on $T$}.
\]
\end{prop}

Considering a particular subclass of nonnegative measures,
we obtain various $1-\alpha$ simultaneous confidence bands.
For example,
\[
 f(t,c) \in \bigl( f(t,\widehat c)
 - \sqrt{\lambda_{12,\alpha}}\,\Vert \psi(t) \Vert_{\Sigma},\infty \bigr)
 \quad\mbox{for all $t\in T$},
\]
and
\[
 \int_{t_0}^t f(t,c)\,dt \in \biggl( \int_{t_0}^t f(t,\widehat c)\,dt
 - \sqrt{\lambda_{12,\alpha}}\,
 \biggl\Vert \int_{t_0}^t \psi(t)\,dt\,\biggr\Vert_{\Sigma},\infty \biggr)
 \quad\mbox{for all $t\in T$}
\]
hold with probabilities greater than or equal to $1-\alpha$.

When $\sigma^2$ is unknown but its unbiased estimator $\widehat\sigma^2$
with $\nu=N-n-1$ degrees of freedom is available,
we can obtain the simultaneous confidence bands by replacing
$\Vert\,\Vert_{\Sigma}$ with $\Vert\,\Vert_{\widehat\Sigma}$
 ($\widehat\Sigma=\widehat\sigma^2\Sigma_0$), and
$\lambda_{12,\alpha}$ with $\lambda_{12,\alpha}'$,
where $\lambda_{12,\alpha}'$ is the upper $\alpha$ quantile of
the distribution of $\lambda_{12}/\sqrt{s}$, $s\sim\chi^2_\nu/\nu$.

\section{Null distributions of the LRT statistics}
\label{sec:null}

\subsection{The volume-of-tubes method}
\label{subsec:tube}

In this subsection, we briefly summarize the volume-of-tubes method.

Historically, the distributions of the orthogonal projection of
zero-mean Gaussian random vectors have been well studied, since
they appear as the null distribution of the test statistic in an order restricted inference.
From the general theory,
the statistics $\lambda_{01}$ and $\lambda_{12}$ in (\ref{lambda})
under the null hypothesis $H_0:c=0$ have the following distribution:
\begin{equation}
\label{chi-bar}
P_{H_0}(\lambda_{01}\ge a,\,\lambda_{12}\ge b) =
 \sum_{i=0}^{n+1} w_i \bar G_i(a) \bar G_{n+1-i}(b),
\end{equation}
where $\bar G_i$ is the upper probability of the chi-square distribution
with $i$ degrees of freedom.
Let $\bar G_0(a)=1$ ($a\le 0$), $0$ ($a>0$).
In addition, the distribution of the LRTs $\beta_{01}$
and $\beta_{12}$ in (\ref{beta}) under $H_0$ is expressed as follows:
\begin{equation}
\label{E-bar}
P_{H_0}(\beta_{01}\ge a,\,\beta_{12}\ge b) =
 \sum_{i=0}^{n+1} w_i \bar B_{\frac{i}{2},\frac{n+1-i+\nu}{2}}(a) \bar B_{\frac{n+1-i}{2},\frac{\nu}{2}}(b),
\end{equation}
where $\bar B_{k,l}$ is the upper probability of the beta distribution
with parameter $(k,l)$.

Note that the coefficients $w_i$ appearing in (\ref{E-bar}) are the same as
those in (\ref{chi-bar}).
They are nonnegative and satisfy $\sum_i w_i=1$.
This means that the distributions of $(\lambda_{01},\lambda_{12})$
and $(\beta_{01},\beta_{12})$ are finite mixture distributions
with the same weights $\{w_i\}$.
The marginal distributions of
$\lambda_{01}$, $\lambda_{12}$, $\beta_{01}$, $\beta_{12}$
can be obtained just by letting $a=-\infty$ or $b=-\infty$.
The finite mixture distribution of the chi-square distributions
in (\ref{chi-bar}) is sometimes referred to as the
chi-bar-square ($\bar\chi^2$) distribution
(\cite{Robertson-etal88}, \cite{Shapiro88}).

When the cone $K_n$ in (\ref{H012}) is polyhedral, that is,
a finite intersection of half spaces,
the weights $\{w_i\}$ can be understood in terms of the internal and
external angles of each face of the cone (\cite{Wynn75}).
Moreover, in the general case where $K_n$ is not polyhedral,
\cite{Takemura-Kuriki97, Takemura-Kuriki02}
proved that the weights $\{w_i\}$ are
expressed as integrals of elementary symmetric polynomials of
principle curvatures of the boundaries of the cone $K_n$.
These integrals are not easy to handle in general.
However, the weights of the two highest degrees and two lowest degrees,
$w_{n+1},w_{n},w_{0}$ and $w_{1}$,
have relatively simple expressions as follows:
\begin{align}
& 
 w_{n+1} = \frac{\vol_{n}(K_n\cap \S^n)}{\Omega_{n+1}}, \quad
 w_{n} = \frac{\vol_{n-1}(\partial K_n\cap \S^n)}{2\,\Omega_{n}},
 \nonumber \\
&
 w_{1} = \frac{\vol^*_{n-1}(\partial K_n^*\cap (\S^n)^*)}{2\,\Omega_{n}}, \quad
 w_{0} = \frac{\vol^*_{n}(K_n^*\cap (\S^n)^*)}{\Omega_{n+1}},
\label{weights}
\end{align}
where
$\partial K_n$ and $\partial K_n^*$ are the boundaries of $K_n$ and $K_n^*$,
\[
 \S^n = \{ x\in \R^{n+1} \mid \Vert x\Vert_{\Sigma^{-1}} =1 \}, \quad
 (\S^n)^* = \{ x\in \R^{n+1} \mid \Vert x\Vert_{\Sigma} =1 \}
\]
are the unit spheres, $\vol_d$ and $\vol^*_d$ are $d$-dimensional
volumes induced by the metrics $\langle\,,\,\rangle_{\Sigma^{-1}}$ and
$\langle\,,\,\rangle_{\Sigma}$, respectively, and
\[
 \Omega_d 
 = \frac{2\pi^{d/2}}{\Gamma(d/2)}
\]
is the volume of the $(d-1)$-dimensional unit sphere in $\R^d$.
The lower dimensional measure is induced by the metric of the ambient space
$\R^{n+1}$.
It is also defined as the Hausdorff measure (\cite{Federer96}).

Moreover, a useful relation
is known as a consequence of the Gauss-Bonnet theorem:
\begin{equation}
\label{gauss-bonnet}
 \sum_{i:\rm odd} w_i = \sum_{i:\rm even} w_i = \frac{1}{2}.
\end{equation}

The distribution of $\beta_{01}$ with $\nu=0$ is interpreted as
the volume formula of a spherical tubular neighborhood as below.
Let $M=K\cap\S^n$ be the intersection of the cone $K\,(=K_n)$ in (\ref{K}) and
the unit sphere.
Define the spherical tube about $M$ with the radius $\theta$:
\[
 \mathrm{Tube}(M,\theta) = \Bigl\{ x\in\S^n \mid
   \min_{y\in M} \mathrm{dist}(x,y) \le \theta \Bigr\}, \quad
 \mathrm{dist}(x,y) = \cos^{-1}\langle x,y\rangle.
\]
Then, because
\[
 \beta_{01} = \frac{\Vert \Pi(\widehat c|K) \Vert^2}{\Vert \widehat c\Vert^2}
 \ge \cos^2\theta
 \ \ \Leftrightarrow\ \ %
 \frac{\widehat c}{\Vert \widehat c \Vert} \in \mathrm{Tube}(M,\theta),
\]
and
$\widehat c/\Vert \widehat c \Vert$ is distributed uniformly on $\S^n$
under $H_0$, we see that
\[
 \frac{\vol_{n}(\mathrm{Tube}(M,\theta))}{\vol_{n}(\S^n)}
 = P_{H_0}(\beta_{01} \ge \cos^2\theta) =
 \sum_{i=0}^{n+1} w_i
 \bar B_{\frac{i}{2},\frac{n+1-i}{2}}(\cos^2\theta).
\]
This is the reason why our methodology is called the tube method
or the volume-of-tubes method.

The volume-of-tubes method has been developed as a tool for approximating
the tail probability of the maximum of a general Gaussian random field
(\cite{Knowles-Siegmund89}, \cite{Sun93}, \cite{Adler-Taylor07}).
This is regarded as a generalization of the distribution
of the projection length of a Gaussian vector onto a convex cone
(\cite{Kuriki-Takemura01}).
This method is also used for the construction of confidence bands
(\cite{Johnstone-Siegmund89}, \cite{Naiman90}).
For the comprehensive survey, see \cite{Kuriki-Takemura09}.

\subsection{Representations for the cones $K_n$ and $K_n^*$}
\label{subsec:parameterization}

In order to evaluate the volumes in (\ref{weights}), we need to
introduce ``local coordinates'' of
the cones $K_n$ in (\ref{K}), $K_n^*$ in (\ref{Ks}),
and their boundaries $\partial K_n$ and $\partial K_n^*$.
This is actually possible by means of representations 
in the theory of Tchebycheff systems.
We consider the following three cases separately:
(i) $T=[a,b]$ (bounded),
(ii) $T=[a,\infty)$, and
(iii) $T=(-\infty,\infty)$.

The two propositions below give representations for the moment cone
$K_n^*$ and its boundary $\partial K_n^*$.
Let
\[
 \psi_n(t) = \begin{cases}
  (1,t,\ldots,t^n)^\top       & (|t|<\infty), \\
  (0,\ldots,0,(\pm 1)^n)^\top & (t=\pm\infty).
\end{cases}
\]
Let
$\R_+=(0,\infty)$, and
\begin{equation}
\label{Delta}
 \Delta_m = \Delta_m(T) = \{ \tau=(\tau_1,\ldots,\tau_m) \in
 (\mathrm{int}\,T)^m \mid \tau_1 < \cdots < \tau_m \}.
\end{equation}
Let $\Delta_0=\emptyset$ formally.
\begin{prop}
\label{prop:Ks-map}
The moment cone $K_{n}^*$ on
(i) $T=[a,b]$, (ii) $T=[a,\infty)$, or (iii) $T=(-\infty,\infty)$
 (when $n=2m$ is even) has the following almost everywhere representations.
Let $a=-\infty$ when $T=(-\infty,\infty)$, and $b=\infty$ when $T=[a,\infty)$
 or $(-\infty,\infty)$.
\begin{align}
K_{n}^*
& = \phi^{(U)}_{n,n} \Bigl( \R_+^{\left[\frac{n+1}{2}\right]+1} \times \Delta_{\left[\frac{n}{2}\right]} \Bigr) \nonumber \\
& = \phi^{(L)}_{n,n} \Bigl( \R_+^{\left[\frac{n}{2}\right]+1} \times \Delta_{\left[\frac{n+1}{2}\right]} \Bigr)
\label{Ks-map}
\end{align}
almost everywhere with respect to the $(n+1)$-dimensional Lebesgue measure,
where
\begin{equation}
\label{upper}
 \phi^{(U)}_{n,l}(\rho,\tau) =
\begin{cases}
 \displaystyle
 \sum_{i=1}^m \rho_{i} \psi_n(\tau_i) + \rho_{m+1} \psi_n(b)
 & (l=2m), \\
 \displaystyle
 \rho_{1} \psi_n(a) + \sum_{i=1}^{m} \rho_{i+1} \psi_n(\tau_i) + \rho_{m+2} \psi_n(b)
 & (l=2m+1),
\end{cases}
\end{equation}
and
\begin{equation}
\label{lower}
 \phi^{(L)}_{n,l}(\rho,\tau) =
\begin{cases}
 \displaystyle
 \rho_1 \psi_n(a) + \sum_{i=1}^m \rho_{i+1} \psi_n(\tau_i)
 & (l=2m), \\
 \displaystyle
 \sum_{i=1}^{m+1} \rho_{i} \psi_n(\tau_i)
 & (l=2m+1).
\end{cases}
\end{equation}
The maps $\phi^{(U)}_{n,n}$ and $\phi^{(L)}_{n,n}$ in
(\ref{Ks-map}) are diffeomorphic. 
\end{prop}
\begin{rem}
\label{rem:lower-upper}
The representations with (\ref{upper}) and (\ref{lower})
are called the upper and lower representations, respectively.
They are coincident when $T=(-\infty,\infty)$ and $n=2m$
 (Definition 3.2 of \cite{Karlin-Studden66}, Section 3 of Chapter II). 
\end{rem}

\begin{rem}
\label{rem:convention}
When $n=1$, $\phi^{(U)}_{n,n}(\rho,\tau)$ is
$\rho_{1} \psi_1(a) + \rho_{2} \psi_1(b)$,
which does not contain the argument $\tau$.
In (\ref{Ks-map}), 
$\phi^{(U)}_{n,n} \bigl( \R_+^{\left[\frac{n+1}{2}\right]+1} \times \Delta_{\left[\frac{n}{2}\right]} \bigr)
=\phi^{(U)}_{1,1} \bigl( \R_+^{2} \times \emptyset \bigr)$ should read as
$\bigl\{\phi^{(U)}_{1,1}(\rho,\tau) \mid \rho\in \R_+^{2}\bigr\}$.
We use this convention in Propositions \ref{prop:Ks-map}--\ref{prop:bK-map}.
\end{rem}

\begin{proof}
The representations of the right-hand sides of (\ref{Ks-map})
for $T=[a,b]$, $[a,\infty)$, and $(-\infty,\infty)$ are provided in
Section 3 of Chapter II, Section 4 of Chapter V, and 
Section 2 of Chapter VI of \cite{Karlin-Studden66}, respectively.
The last case of $T=(-\infty,\infty)$ is stated in terms of periodic functions.
Each of the upper and lower representations is the unique representation
when $\rho_i>0$ for all $i$, and all of $a$, $\tau_i$ and $b$ are distinct.

Although the representations given by \cite{Karlin-Studden66} include
the cases where $\rho_i=0$ for some $i$, and some of $a$, $\tau_i$ and $b$
take the same value, we can ignore them because the images of the maps
$\phi^{(U)}_{n,n}$ and $\phi^{(L)}_{n,n}$ in such cases
are at most $n$-dimensional.

The maps $\phi^{(U)}_{n,n}$ and $\phi^{(L)}_{n,n}$ are one-to-one and
obviously differentiable, that is, diffeomorphic.
\end{proof}

\begin{prop}
\label{prop:bKs-map}
Suppose that $n\ge 2$.
Let $\Delta_m$ be defined in (\ref{Delta}).
Let $\phi^{(U)}_{n,l}$ and $\phi^{(L)}_{n,l}$ be defined in
(\ref{upper}) and (\ref{lower}).
The boundary of the moment cone $\partial K_{n}^*$
has the following almost everywhere representation.
Let again $a=\inf T$ and $b=\sup T$.
\\
(i), (ii) When $T=[a,b]$ or $[a,\infty)$,
\begin{align}
\partial K_{n}^*
=&
 \phi^{(L)}_{n,n-1} \Bigl( \R_+^{\left[\frac{n-1}{2}\right]+1} \times \Delta_{\left[\frac{n}{2}\right]} \Bigr) \nonumber \\
 & \sqcup
 \phi^{(U)}_{n,n-1} \Bigl( \R_+^{\left[\frac{n}{2}\right]+1} \times \Delta_{\left[\frac{n-1}{2}\right]} \Bigr),
\label{bKs-map}
\end{align}
(iii) when $T=(\infty,\infty)$ and $n=2m$ is even,
\begin{align}
\partial K_{n}^*
=&
 \phi^{(L)}_{n,n-1} \Bigl( \R_+^{m} \times \Delta_{m} \Bigr)
\label{bKs-map'}
\end{align}
almost everywhere with respect to the $n$-dimensional Hausdorff measure,
where $\sqcup$ means a disjoint union.
The maps $\phi^{(U)}_{n,n-1}$ and $\phi^{(L)}_{n,n-1}$
in (\ref{bKs-map}) and (\ref{bKs-map'}) are diffeomorphic.
\end{prop}

\begin{proof}
The general forms of the one-to-one representations
for $T=[a,b]$, $[a,\infty)$, and $(-\infty,\infty)$ are provided in
Section 2 of Chapter II, Section 4 of Chapter V, and Section 5 of Chapter VI
of \cite{Karlin-Studden66}, respectively.
The last case of $T=(-\infty,\infty)$ is stated in terms of periodic functions.
Picking up the terms whose images are $n$-dimensional, we have
(\ref{bKs-map}) and (\ref{bKs-map'}).
The second component in (\ref{bKs-map}) disappears
in (\ref{bKs-map'}) because when $T=(-\infty,\infty)$ and $n=2m$,
$\psi_n(a)=\psi_n(b)=(0,\ldots,0,1)^\top$.
\end{proof}

The following two propositions give representations for the positive polynomial
cone $K_{n}$ and its boundary $\partial K_{n}$.
\begin{prop}
\label{prop:K-map}
The positive polynomial cone $K_{n}$ has
the following almost everywhere representation:
\[
 K_{n} = \varphi_n \Bigl( \R_+^2 \times \Delta_{n-1} \Bigr)
\]
almost everywhere with respect to the $(n+1)$-dimensional Lebesgue measure.
Here, the function $\varphi_n (\alpha,\gamma)\in\R^{n+1}$ with
$\alpha=(\alpha_1,\alpha_2)\in \R_+^2$
and
$\gamma=(\gamma_1,\ldots,\gamma_{n-1})\in\Delta_{n-1}$
is the coefficient vector of the polynomial
$p_n(t;\alpha,\gamma) = \varphi_n (\alpha,\gamma)^\top\psi_n(t)$
in $t$ defined below:
\\
(i) When $T=[a,b]$,
\[
 p_{n}(t;\alpha,\gamma) =
\begin{cases}
 \displaystyle
  \alpha_1 \prod_{j=1}^{m} (t-\gamma_{2j-1})^2
+ \alpha_2 (t-a)(b-t) \prod_{j=1}^{m-1} (t-\gamma_{2j})^2 & (n=2m), \\
 \displaystyle
  \alpha_1 (t-a) \prod_{j=1}^{m} (t-\gamma_{2j})^2
+ \alpha_2 (b-t)\prod_{j=1}^{m} (t-\gamma_{2j-1})^2 & (n=2m+1),
\end{cases}
\]
(ii) when $T=[a,\infty)$,
\[
 p_{n}(t;\alpha,\gamma) =
\begin{cases}
 \displaystyle
  \alpha_1 \prod_{j=1}^{m} (t-\gamma_{2j-1})^2
+ \alpha_2 (t-a) \prod_{j=1}^{m-1} (t-\gamma_{2j})^2 & (n=2m), \\
 \displaystyle
  \alpha_1 (t-a) \prod_{j=1}^{m} (t-\gamma_{2j})^2
+ \alpha_2 \prod_{j=1}^{m} (t-\gamma_{2j-1})^2 & (n=2m+1),
\end{cases}
\]
(iii) when $T=(-\infty,\infty)$ and $n=2m$,
\[
 p_{n}(t;\alpha,\gamma) =
  \alpha_1 \prod_{j=1}^{m} (t-\gamma_{2j-1})^2
+ \alpha_2 \prod_{j=1}^{m-1} (t-\gamma_{2j})^2.
\]
The map $\varphi_n$ is a diffeomorphism.
Here, we use the convention $\prod_{j=1}^0=1$.
\end{prop}
\begin{proof}
The representations of the positive polynomials on $T=[a,b]$, $[a,\infty)$,
and $(-\infty,\infty)$ whose orders are exactly $n$
are provided in
Section 10 of Chapter II,
Section 9 of Chapter V, and
Section 9 of Chapter VI
of \cite{Karlin-Studden66}, respectively.
They are unique representations when $\alpha_1,\alpha_2>0$ and $a$, $b$, $\gamma_i$'s
are distinct.

Because the contributions of the positive polynomials of order $n$
with $\alpha_1=0$ or $\alpha_2=0$
and the positive polynomials of order less than $n$ are $n$-dimensional
at most, we do not need to take them into account.

The uniqueness of the representation of $p_n$ implies that
the map $\varphi_n$ is one-to-one.  It is obviously differentiable and hence
diffeomorphic.
\end{proof}

\begin{prop}
\label{prop:bK-map}
Suppose that $n\ge 2$.
The boundary of the positive polynomial cone $\partial K_n$
has the almost everywhere representation below.
Define the functions
$\varphi_n^{(i)}(\alpha,\gamma,\widetilde\gamma)\in\R^{n+1}$
with $n-i\ge 1$,
$\alpha\in\R_+^2$, $\gamma\in\Delta_{n-1-i}$, and $\widetilde\gamma\in\mathrm{int}\,T$,
by the coefficient vectors of polynomials as
\[
 (t-\widetilde\gamma)^i p_{n-i}(t;\alpha,\gamma)
 = \varphi_n^{(i)}(\alpha,\gamma,\widetilde\gamma)^\top\psi_n(t)
 \quad (i=1,2).
\]
Define the function $\varphi_2^{(2)}(\alpha_1,\widetilde\gamma)\in\R^3$
with $\alpha_1\in\R_+$ and $\widetilde\gamma\in\mathrm{int}\,T$
by the coefficient vector of a polynomial as
\[
 (t-\widetilde\gamma)^2 \times \alpha_1 = \varphi_2^{(2)}(\alpha_1,\widetilde\gamma)^\top\psi_2(t).
\]
\\
(i) When $T=[a,b]$,
\begin{align}
\partial K_{n} =
 & \begin{cases}
 \varphi_n^{(2)} \Bigl( \R_+^2 \times \Delta_{n-3} \times T \Bigr) & (n\ge 3),
 \\
 \varphi_2^{(2)} \Bigl( \R_+ \times T \Bigr) & (n=2)
 \end{cases} \nonumber \\
 & \sqcup \varphi_n^{(1)} \Bigl( \R_+^2 \times \Delta_{n-2}, a \Bigr)
 \nonumber \\
 & \sqcup \left\{ -\varphi_n^{(1)}
 \Bigl( \R_+^2 \times \Delta_{n-2}, b \Bigr) \right\},
\label{bK-map}
\end{align}
(ii) when $T=[a,\infty)$,
\begin{align}
\partial K_{n} =
 & \begin{cases}
 \varphi_n^{(2)} \Bigl( \R_+^2 \times \Delta_{n-3} \times T \Bigr) & (n\ge 3),
 \\
 \varphi_2^{(2)} \Bigl( \R_+ \times T \Bigr) & (n=2)
 \end{cases} \nonumber \\
 & \sqcup \varphi_n^{(1)} \Bigl( \R_+^2 \times \Delta_{n-2}, a \Bigr)
 \nonumber \\
 & \sqcup \varphi_{n-1} \Bigl( \R_+^2 \times \Delta_{n-2} \Bigr),
\label{bK-map'}
\end{align}
(iii) when $T=(-\infty,\infty)$ and $n=2m$ is even,
\begin{align}
\partial K_{n} =
 & \begin{cases}
 \varphi_n^{(2)} \Bigl( \R_+^2 \times \Delta_{n-3} \times T \Bigr) & (n\ge 3),
 \\
 \varphi_2^{(2)} \Bigl( \R_+ \times T \Bigr) & (n=2)
 \end{cases}
\label{bK-map''}
\end{align}
almost everywhere with respect to the $n$-dimensional Hausdorff measure,
where $\sqcup$ means a disjoint union.
The maps $\varphi_n$, $\varphi_n^{(1)}(\cdot,a)$, $\varphi_n^{(1)}(\cdot,b)$,
and $\varphi_n^{(2)}$ are diffeomorphisms.
\end{prop}

\begin{proof}
(i) The case of $T=[a,b]$.
The boundary of the positive polynomial cone $K_n$ is proved to consist of
the positive polynomials of order $n$ (at most) that have zeros on $T$.
For the almost everywhere representation,
we need only polynomials of the highest degree.
Hence, we can consider only the following three types:
$(t-\widetilde\gamma)^2 p_{n-2}(t;\alpha,\gamma)$
 ($\widetilde\gamma\in\mathrm{int}\,T$),
$(t-a) p_{n-1}(t;\alpha,\gamma)$, and
$(b-t) p_{n-1}(t;\alpha,\gamma)$
with $\alpha_1,\alpha_2>0$.
The above three types have no intersection,
and (\ref{bK-map}) follows.

(ii) The case of $T=[a,\infty)$.
The boundary of the positive polynomial cone $K_n$ is proved to consist of
the positive polynomials of order $n$ (at most) that have zeros on $T$
and the positive polynomials of order $n-1$ (at most).
For the almost everywhere representation,
we can consider only the following three types:
$(t-\widetilde\gamma)^2 p_{n-2}(t;\alpha,\gamma)$
 ($\widetilde\gamma\in\mathrm{int}\,T$),
$(t-a) p_{n-1}(t;\alpha,\gamma)$, and
$p_{n-1}(t;\alpha,\gamma)$ with $\alpha_1,\alpha_2>0$.
These three types have no intersection,
and (\ref{bK-map'}) follows.

(iii) The case of $T=(-\infty,\infty)$.
The boundary of the positive polynomial cone $K_n$ is proved to consist of
the positive polynomials of order $n$ (at most) that have zeros on $T$
and the positive polynomials of order $n-2$ (at most).
For the almost everywhere representation,
we can consider only the case
$(t-\widetilde\gamma)^2 p_{n-2}(t;\alpha,\gamma)$ ($\widetilde\gamma\in T$),
and (\ref{bK-map''}) follows.
\end{proof}

\subsection{Volume formulas and the weights}
\label{subsec:volume}

The diffeomorphic maps appearing in Propositions 
\ref{prop:Ks-map}--\ref{prop:bK-map}
are homogeneous functions with respect to
 their first arguments $\rho$ and $\alpha$.
Therefore, by restricting the length of the first argument, we can construct
almost everywhere representations for the intersections with the unit sphere.
For example, $\phi^{(U)}_{n,n}(r\rho,\tau)=r \phi^{(U)}_{n,n}(\rho,\tau)$
for a constant $r>0$, and we have
\[
 K_{n}^*\cap(\S^n)^*
 = \bar\phi^{(U)}_{n,n}
 \Bigl( \S_+^{\left[\frac{n+1}{2}\right]} \times
 \Delta_{\left[\frac{n}{2}\right]} \Bigr)
 \quad \mathrm{a.e.},
\]
where
\[
 \bar\phi^{(U)}_{n,l}(\rho,\tau) =
 \phi^{(U)}_{n,l}(\rho,\tau)/ \Vert \phi^{(U)}_{n,l}(\rho,\tau) \Vert_{\Sigma},
\]
and
\begin{equation}
\label{S+}
 \S_+^m = \Bigl\{ \rho =(\rho_i) \in \R^{m+1} \mid \sum\rho_i^2 = 1,\ \rho_i > 0 \Bigr\}.
\end{equation}
Define
\begin{align*}
\bar\phi^{(L)}_{n,l}(\rho,\tau) &=
 \phi^{(L)}_{n,l}(\rho,\tau)/ \Vert \phi^{(L)}_{n,l}(\rho,\tau) \Vert_{\Sigma},
 \\
\bar\varphi_{n}(\alpha,\gamma) &=
 \varphi_{n}(\alpha,\gamma)/ \Vert \varphi_{n}(\alpha,\gamma) \Vert_{\Sigma^{-1}},
 \\
\bar\varphi^{(i)}_{n}(\alpha,\gamma,\widetilde\gamma) &=
 \varphi^{(i)}_{n}(\alpha,\gamma,\widetilde\gamma)/ \Vert \varphi^{(i)}_{n}(\alpha,\gamma,\widetilde\gamma) \Vert_{\Sigma^{-1}}
 \quad (i=1,2), \\
\bar\varphi^{(2)}_{2}(\widetilde\gamma) &=
 \varphi^{(2)}_{2}(1,\widetilde\gamma)/ \Vert \varphi^{(2)}_{2}(1,\widetilde\gamma) \Vert_{\Sigma^{-1}}
\end{align*}
similarly.

In the proposition below, let $\theta=(\theta_i)\in\Theta_m$ be
the local coordinates of $\S_+^m$ in (\ref{S+}).
For example,
$\rho=\rho(\theta)=\bigl(\theta_1,\ldots,\theta_m,\sqrt{1-\sum\theta_i^2}\bigr)$,
$\theta\in\Theta_m=\R_+^m$.
Another example is the polar coordinates
$\rho(\theta) = (\rho_i(\theta))$, $\theta\in\Theta_m = (0,\pi/2)^m$
with
$\rho_1(\theta) = \cos\theta_1$,
$\rho_i(\theta) = \cos\theta_i \prod_{j=1}^{i-1}\sin\theta_j$
 ($i=2,\ldots,m$), and
$\rho_{m+1}(\theta) = \prod_{j=1}^{m}\sin\theta_j$.

\begin{prop}
\label{prop:ws}
Let $\xi=(\theta,\tau)$ and
$d\xi=\prod d\theta_i\prod d\tau_i$ be the Lebesgue measure.
($\xi$ may consist of either $\theta$ or $\tau$ only
when the other does not appear in the integrand.)
Write $\rho=\rho(\theta)$ for simplicity.
\begin{align*}
\vol^*_{n}( K_n^*\cap (\S^n)^*) 
& = \int_{\Theta_{\left[\frac{n+1}{2}\right]}\times\Delta_{\left[\frac{n}{2}\right]}} \det \left\{
\biggl(\frac{\partial \bar\phi^{(U)}_{n,n}(\rho,\tau)}{\partial\xi}\biggr)^\top
\Sigma
\biggl(\frac{\partial \bar\phi^{(U)}_{n,n}(\rho,\tau)}{\partial\xi}\biggr)
\right\}^{\frac{1}{2}} d\xi
 \\
& = \int_{\Theta_{\left[\frac{n}{2}\right]}\times\Delta_{\left[\frac{n+1}{2}\right]}} \det \left\{
\biggl(\frac{\partial \bar\phi^{(L)}_{n,n}(\rho,\tau)}{\partial\xi}\biggr)^\top
\Sigma
\biggl(\frac{\partial \bar\phi^{(L)}_{n,n}(\rho,\tau)}{\partial\xi}\biggr)
\right\}^{\frac{1}{2}} d\xi,
\end{align*}
and when $n\ge 2$,
\begin{align}
\vol^*_{n-1}( \partial & K_n^*\cap (\S^n)^*) \nonumber\\
=& \int_{\Theta_{\left[\frac{n-1}{2}\right]}\times\Delta_{\left[\frac{n}{2}\right]}} \det \left\{
\biggl(\frac{\partial \bar\phi^{(L)}_{n,n-1}(\rho,\tau)}{\partial\xi}\biggr)^\top
\Sigma
\biggl(\frac{\partial \bar\phi^{(L)}_{n,n-1}(\rho,\tau)}{\partial\xi}\biggr)
\right\}^{\frac{1}{2}} d\xi \nonumber \\
 &+\int_{\Theta_{\left[\frac{n}{2}\right]}\times\Delta_{\left[\frac{n-1}{2}\right]}} \det \left\{
\biggl(\frac{\partial \bar\phi^{(U)}_{n,n-1}(\rho,\tau)}{\partial\xi}\biggr)^\top
\Sigma
\biggl(\frac{\partial \bar\phi^{(U)}_{n,n-1}(\rho,\tau)}{\partial\xi}\biggr)
\right\}^{\frac{1}{2}} d\xi \nonumber \\
& \hspace*{15em}
 (\mbox{if\,\ $T=[a,b]$ or $[a,\infty)$}).
\label{w1}
\end{align}
The second term in the right-hand side of (\ref{w1}) is not needed
when $T=(-\infty,\infty)$. 
\end{prop}

\begin{prop}
\label{prop:w}
Let $\zeta=(\theta,\gamma,\widetilde\gamma)$ and
$d\zeta=d\theta\prod d\gamma_i\,d\widetilde\gamma$ be the Lebesgue measure.
(Some of $\theta,\gamma,\widetilde\gamma$ may not be included in $\zeta$
if they do not appear in the integrand.)
Let $\alpha=(\cos\theta,\sin\theta)$.
\[
\vol_{n}( K_n\cap \S^n)
 = \int_{(0,\frac{\pi}{2})\times\Delta_{n-1}} \det \left\{
\biggl(\frac{\partial \bar\varphi_{n}(\alpha,\gamma)}{\partial\zeta}\biggr)^\top
\Sigma^{-1}
\biggl(\frac{\partial \bar\varphi_{n}(\alpha,\gamma)}{\partial\zeta}\biggr)
\right\}^{\frac{1}{2}} d\zeta,
\]
and when $n\ge 2$,
\begin{align}
\vol_{n-1}( \partial & K_n\cap \S^n) \nonumber \\
=&
 \begin{cases}
 \displaystyle
 \int_{(0,\frac{\pi}{2})\times\Delta_{n-3}\times T} \det \left\{
\biggl(\frac{\partial \bar\varphi^{(2)}_{n}(\alpha,\gamma,\widetilde\gamma)}{\partial\zeta}\biggr)^\top
\Sigma^{-1}
\biggl(\frac{\partial \bar\varphi^{(2)}_{n}(\alpha,\gamma,\widetilde\gamma)}{\partial\zeta}\biggr)
\right\}^{\frac{1}{2}} d\zeta & (n\ge 3), \\
 \displaystyle
 \int_{T} \left\{
\biggl(\frac{\partial \bar\varphi^{(2)}_{2}(\widetilde\gamma)}{\partial\widetilde\gamma}\biggr)^\top
\Sigma^{-1}
\biggl(\frac{\partial \bar\varphi^{(2)}_{2}(\widetilde\gamma)}{\partial\widetilde\gamma}\biggr)
\right\}^{\frac{1}{2}} d\widetilde\gamma & (n=2)
 \end{cases} \nonumber \\
&+
 \int_{(0,\frac{\pi}{2})\times\Delta_{n-2}} \det \left\{
\biggl(\frac{\partial \bar\varphi^{(1)}_{n}(\alpha,\gamma,a)}{\partial\zeta}\biggr)^\top
\Sigma^{-1}
\biggl(\frac{\partial \bar\varphi^{(1)}_{n}(\alpha,\gamma,a)}{\partial\zeta}\biggr)
\right\}^{\frac{1}{2}} d\zeta \nonumber \\
& \hspace*{17.5em}
 (\mbox{if\,\ $T=[a,b]$ or $[a,\infty)$}) \nonumber \\
&+
 \int_{(0,\frac{\pi}{2})\times\Delta_{n-2}} \det \left\{
\biggl(\frac{\partial \bar\varphi^{(1)}_{n}(\alpha,\gamma,b)}{\partial\zeta}\biggr)^\top
\Sigma^{-1}
\biggl(\frac{\partial \bar\varphi^{(1)}_{n}(\alpha,\gamma,b)}{\partial\zeta}\biggr)
\right\}^{\frac{1}{2}} d\zeta \nonumber \\
& \hspace*{17.5em}
 (\mbox{if\,\ $T=[a,b]$}) \nonumber \\
&+
 \int_{(0,\frac{\pi}{2})\times\Delta_{n-2}} \det \left\{
\biggl(\frac{\partial \bar\varphi_{n-1}(\alpha,\gamma)}{\partial\zeta}\biggr)^\top
\Sigma^{-1}
\biggl(\frac{\partial \bar\varphi_{n-1}(\alpha,\gamma)}{\partial\zeta}\biggr)
\right\}^{\frac{1}{2}} d\zeta \nonumber \\
& \hspace*{17.5em}
 (\mbox{if\,\ $T=[a,\infty)$}).
\label{wn}
\end{align}
In the right-hand side of (\ref{wn}),
the second term is not needed for $T=(-\infty,\infty)$,
the third term is not needed for $T=[a,\infty)$ and $(-\infty,\infty)$,
the fourth term is not needed for $T=[a,b]$ and $(-\infty,\infty)$.
\end{prop}

Substituting the volumes obtained in
Propositions \ref{prop:ws} and \ref{prop:w} into (\ref{weights}),
we get $w_{n+1}$, $w_{n}$, $w_{0}$ and $w_{1}$.
Combined with the Gauss-Bonnet theorem (\ref{gauss-bonnet}),
all weights $\{w_i\}$ for $n\le 4$ are obtained as follows.
\[
(w_0,\ldots,w_{n+1})=
\begin{cases}
 \bigl(w_0,\frac{1}{2},\frac{1}{2}-w_0\bigr) =
 \bigl(\frac{1}{2}-w_{n+1},\frac{1}{2},w_{n+1}\bigr)
 & (n=1), \\
 \bigl(\frac{1}{2}-w_n,w_1,w_n,\frac{1}{2}-w_1\bigr)
 & (n=2), \\
 \bigl(w_0,w_1,\frac{1}{2}-w_0-w_{n+1},w_n,w_{n+1}\bigr) & (n=3), \\
 \bigl(w_0,w_1,\frac{1}{2}-w_0-w_n,\frac{1}{2}-w_1-w_{n+1},w_n,w_{n+1}\bigr)
& (n=4).
\end{cases}
\]

For $n>4$, some of the weights are undetermined.
However, thanks to the Gauss-Bonnet theorem (\ref{gauss-bonnet}),
and noting that $\bar G_i(a)$ and $\bar B_{\frac{i}{2},\frac{n+1-i+\nu}{2}}(a)$
are increasing in $i$, and that
$\bar G_{n+1-i}(b)$ and $\bar B_{\frac{n+1-i}{2},\frac{\nu}{2}}(b)$ are
decreasing in $i$, we have upper and lower bounds for the marginal
distributions of (\ref{chi-bar}) and (\ref{E-bar}).
For example, the bounds for $\lambda_{01}$ and $\lambda_{12}$ are given by
\begin{align*}
 \sum_{i=0}^{n+1} u_i \bar G_i(a) & \le P_{H_0}(\lambda_{01}\ge a) \le
 \sum_{i=0}^{n+1} v_i \bar G_i(a), \\
 \sum_{i=0}^{n+1} v_i \bar G_{n+1-i}(a) & \le P_{H_0}(\lambda_{12}\ge a) \le
 \sum_{i=0}^{n+1} u_i \bar G_{n+1-i}(a),
\end{align*}
where
\begin{align*}
(u_0,&\ldots,u_{n+1}) \\ &=
\begin{cases}
 \bigl(w_0,w_1,\frac{1}{2}-w_0-w_{n+1},\frac{1}{2}-w_1-w_{n},\underbrace{0,\ldots,0}_{n-4},w_n,w_{n+1}\bigr)
& (\mbox{$n$\,:\,odd}), \\
 \bigl(w_0,w_1,\frac{1}{2}-w_0-w_{n},\frac{1}{2}-w_1-w_{n+1},\underbrace{0,\ldots,0}_{n-4},w_n,w_{n+1}\bigr)
& (\mbox{$n$\,:\,even}),
\end{cases} \\
(v_0,&\ldots,v_{n+1}) \\ &=
\begin{cases}
 \bigl(w_0,w_1,\underbrace{0,\ldots,0}_{n-4},\frac{1}{2}-w_1-w_{n},
 \frac{1}{2}-w_0-w_{n+1},w_n,w_{n+1}\bigr)
& (\mbox{$n$\,:\,odd}), \\
 \bigl(w_0,w_1,\underbrace{0,\ldots,0}_{n-4},\frac{1}{2}-w_0-w_{n},\frac{1}{2}-w_1-w_{n+1},w_n,w_{n+1}\bigr)
& (\mbox{$n$\,:\,even}).
\end{cases}
\end{align*}

Moreover, since $\bar G_i(a)=o(\bar G_{n+1}(a))$ as $a\to\infty$ for $i<n+1$,
the tail probabilities of $\lambda_{01}$ and $\lambda_{12}$ have asymptotic
expressions
\[
 P_{H_0}(\lambda_{01}\ge a) \, \sim \, w_{n+1} \bar G_{n+1}(a), \quad
 P_{H_0}(\lambda_{12}\ge a) \, \sim \, w_{0} \bar G_{n+1}(a)
\]
as $a\to\infty$.

\section{Computational aspects}
\label{sec:computation}

\subsection{A numerical procedure for MLE}
\label{subsec:programming}

To obtain the LRT statistics $\lambda_{01}$ and $\lambda_{12}$ in
(\ref{lambda}), we need to perform the orthogonal projection onto
the positive polynomial cone $K$.  For this purpose,
the following symmetric cone programming technique is useful.
In this subsection, we treat only the case of $T=[a,b]$ with finite $a,b$.   
However, the technique explained here is easily extended to the other cases.

The positive polynomial $p_{n}(t)$ of degree $n$ on the set $T$ is
characterized in Proposition \ref{prop:K-map}.  
This is a unique representation.
Admitting the redundancy of the parameters, this polynomial is
rewritten as
\begin{equation}
\label{markov-lukacs}
p_{n} (t) = \begin{cases}
 \displaystyle
  \psi_m(t)^\top Q_1 \psi_m(t)
 + (t-a)(b-t) \psi_{m-1}(t)^\top Q_2 \psi_{m-1}(t) & (n=2m), \\
 \displaystyle
   (t-a) \psi_m(t)^\top Q_1 \psi_m(t)
 + (b-t) \psi_m(t)^\top Q_2 \psi_m(t) & (n=2m+1),
\end{cases}
\end{equation}
where $Q_1$ and $Q_2$ are symmetric positive semi-definite matrices. 
This polynomial (\ref{markov-lukacs}) is obviously nonnegative on $T=[a,b]$.
Conversely, the polynomial $p_{n}(t)$ in Proposition \ref{prop:K-map}
can be written as (\ref{markov-lukacs}).
This representation is sometimes referred to as
the Markov-Lukacs theorem (\cite{Nesterov00}).

By arranging the terms, the polynomial $p_{n}(t)$ in (\ref{markov-lukacs})
can be written as $p_{n}(t)=e(Q_1,Q_2)^\top \psi_n(t)$, where $e(Q_1,Q_2)$
is a $(n+1)$-dimensional column vector depending on $Q_1$ and $Q_2$.
Using this representation, the orthogonal projection of a given vector
$\widehat c$ onto the positive polynomial cone $K$ is formalized as
the optimization problem below:

\begin{quote}
\begin{tabular}{lll}
\texttt{maximize}   & $-d$ & \\[1mm]
\texttt{subject to} & $d \ge \Vert \widehat c - c \Vert_{\Sigma^{-1}}$ &
 (quadratic cone restriction) \\[1mm]
                    & $c = e(Q_1,Q_2)$ &
 (linear restriction) \\[1mm]
                    & $Q_1, Q_2 \succeq 0$ &
 (PSD cone restriction)
\end{tabular}
\end{quote}
This is an optimization problem with quadratic cone, linear, and
positive semi-definite (PSD) cone restrictions.
This can be solved in the framework of symmetric cone programming.
Several public software programs are available (e.g., SeDuMi by \cite{Sturm99}).

Figure \ref{fig:proj} shows an example of orthogonal projection.
Let $K$ be the positive polynomial cone of order $n=3$ on the set
$T=[a,b]=[0,1]$.
Under the metric $\Vert\,\Vert_{\Sigma^{-1}}$ with
$\Sigma = \left((i+j-1)^{-1}\right)^{-1}_{1\le i,j\le 4}$,
the orthogonal projection of
$f(t; \widehat c) = 0.5 t -1.5 t^2 + t^3$ 
onto $K$ is given by
$f(t; \widehat c_K) = 0.0258 + 0.5151 t -1.4891 t^2 + 1.0086 t^3$. 
In Figure \ref{fig:proj},
$f(t; \widehat c)$ is depicted as a dashed line (- - -),
and the projection $f(t; \widehat c_K)$ is depicted as a solid line (------).
\begin{figure}[h]
\begin{center}
\vspace*{-13.5mm}
\scalebox{0.4}{\rotatebox{-90}{\includegraphics{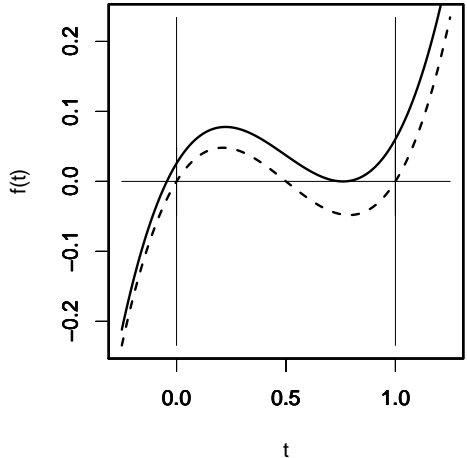}}}
\end{center}
\vspace*{-10mm}
\caption{Projection onto the positive polynomial cone $K_{3}$.}
\label{fig:proj}
\end{figure}

\subsection{Analysis of growth curve data: An example}
\label{subsec:example}

In this subsection, we analyze
growth curve data cited in \cite{Potthoff-Roy64}.
The dataset consists of a certain measurement
on dental study for 11 girls and 16 boys at ages 8, 10, 12, and 14 years.

In our study, let $t$ be the age minus 11
for stabilizing numerical calculations.
The measurements of the individual $h$ at the age $t+11$ in the girl and boy
groups are denoted by $x_{0ht}$ and $x_{1ht}$, respectively.
For modeling the difference of the profiles (mean vectors) of two groups,
we assume the multivariate normal model:
\begin{align}
& x_{0ht} = \mu_t + \varepsilon_{0th}, \quad
  h=1,\ldots,n_0\,(=11), \nonumber \\
& x_{1ht} = \mu_t + f(t;c) + \varepsilon_{1th}, \quad
  h=1,\ldots,n_1\,(=16),
\label{model}
\end{align}
with
\[
 f(t;c) = c_0 + c_1 t + c_2 t^2 + c_3 t^3,
\]
where $\varepsilon_{jh} = (\varepsilon_{jht})_{t\in\{-3,-1,1,3\}}$ ($j=0,1$)
are independent Gaussian error vectors with mean zero.
For the covariance matrices, we assume the intraclass correlation structure
\begin{equation}
\label{intraclass}
 \Sigma_j = \Cov\bigl(\varepsilon_{jh},\varepsilon_{jh}\bigr)
 = \tau_j \{ (1-\rho_j) I + \rho_j J \} \quad (j=0,1),
\end{equation}
where $J$ is the $4\times 4$ matrix with all entries 1, and
$\tau_j$ and $\rho_j$ are unknown parameters.
The model (\ref{intraclass}) is widely used covariance structure
in the analysis of growth curves and repeated measurements
 (\cite{Crowder-Hand90}, \cite{Kato-etal10}).

Under the model (\ref{model}) with (\ref{intraclass}),
the MLEs are calculated as
%
\[
 \widehat c = (\widehat c_0,\widehat c_1,\widehat c_2,\widehat c_3)^\top
 = (2.053, 0.551, 0.0536, -0.0301)^\top,
\]
and
$\widehat\tau_0=4.469$, $\widehat\rho_0=0.868$,
$\widehat\tau_1=5.147$, $\widehat\rho_1=0.479$.
If $\Sigma_0$ and $\Sigma_1$ are known, $\widehat c$ is distributed as the
normal distribution with covariance matrix $\Sigma=(F^\top V^{-1} F)^{-1}$,
where $V= n_0^{-1}\Sigma_0 + n_1^{-1}\Sigma_1$ and
$F=(t^i)_{t\in\{-3,-1,1,3\},\,0\le i\le 3}$ is the design matrix.
The MLE of $\Sigma$ is obtained as
\[
\widehat\Sigma =
\begin{pmatrix}
 0.649   & 0       & -0.0173 & 0 \\
 0       & 0.140   & 0       & -0.0157 \\
 -0.0173 & 0       & 0.00345 & 0 \\
 0       & -0.0157 & 0       & 0.00192
\end{pmatrix}.
\]
In the following, we treat $\widehat\Sigma$ as the true value,
and suppose the statistic $\widehat c$ to be a Gaussian vector
with mean $c$ and covariance matrix $\widehat\Sigma$
as an approximating analysis.

Let us focus on the whole period from ages 8 to 14 years, that is, 
$T=[-3,3]$, and consider the positivity on the set $T$.
The hierarchical hypotheses in (\ref{H012}) are
$H_0$\,:\,$f(t;c)\equiv 0$ ($c=0$),
$H_1$\,:\,$f(t;c)$ is a positive polynomial on $T$ ($c\in K_3$), and
$H_2$\,:\,$f(t;c)$ is unrestricted ($c\in\R^{3+1}$). 
Since $f(t;\widehat c\,)$ is already positive on $T$,
the orthogonal projection $\widehat c_K$ is $\widehat c$ itself,
and the LRT statistic for testing $H_0$ against $H_1$
is $\lambda_{01} = \Vert \widehat c_K \Vert_{\widehat\Sigma^{-1}}^2 =
 \Vert \widehat c \Vert_{\widehat\Sigma^{-1}}^2 = 19.293$. 
This looks highly significant because the $p$-value
referring to the chi-square distribution with 4 degrees of freedom
is already $0.000688$. 
Actually, by means of Propositions \ref{prop:ws} and \ref{prop:w},
the weights for the distribution of $\lambda_{01}$ are
%
\[
 (w_0,w_1,w_2,w_3,w_4) = (0.0072, 0.0657, 0.2416, 0.4343, 0.2512),
\]
and the $p$-value for $\lambda_{01}$ is obtained as 0.000293. 
We conclude that the growth curve of the boy group
is always beyond that of the girl group. 

Then, what about the growth rates of the two groups?  Is
the growth rate of the boy group always greater than that of the girl group?
In order to confirm this hypothesis,
let us take the differential of $f(t;\widehat c\,)$:
\[
 f'(t;\widehat c\,) = \widehat c_1 + 2\,\widehat c_2\,t + 3\,\widehat c_3\,t^2
 =  f(t;\widehat d\,),
\]
where
\[
 \widehat d = L\,\widehat c = (0.551, 0.107, -0.0902)^\top, \quad
 L =
 \begin{pmatrix}
  0 & 1 & 0 & 0 \\
  0 & 0 & 2 & 0 \\
  0 & 0 & 0 & 3
 \end{pmatrix}.
\]
We suppose that $\widehat d$ is distributed as the normal distribution
$N_3(d,L\widehat\Sigma L^\top)$, $d=L c$.
Here again, we consider the hierarchical hypotheses in (\ref{H012}) that
$H_0$\,:\,$f'(t;c)\equiv 0$ ($d=0$),
$H_1$\,:\,$f'(t;c)$ is a positive polynomial on $T$ ($d\in K_2$), and
$H_2$\,:\,$f'(t;c)$ is unrestricted ($d\in\R^{2+1}$). 
Since $f'(-3;\widehat c\,)=-0.582<0<f'(0;\widehat c\,)=\widehat c_1\,(=0.551)$,
$f'(t;\widehat c\,)$ is not a positive polynomial on $T=[-3,3]$.
The orthogonal projection of $\widehat d$ onto $K$ under the metric
$\langle , \rangle_{(L\widehat\Sigma L^\top)^{-1}}$ is
\[
 \widehat d_K = (0.348, 0.0776, -0.0128)^\top.
\]
The LRT statistics for testing $H_0$ against $H_1$, and
for testing $H_1$ against $H_2$ are obtained as
$\lambda_{01}=9.293$ and $\lambda_{12}=0.417$, respectively.
The weights are computed as
\[
 (w_0,w_1,w_2,w_3) = (0.3318, 0.4792, 0.168, 0.0208).
\]
Using these weights, the $p$-values for $\lambda_{01}$ and $\lambda_{12}$
are calculated as $0.00324$ and $0.787$, respectively.
Thus, the hypothesis that $f'$ is a positive polynomial is accepted, and
the hypothesis that $f'\equiv 0$ is rejected at the 1\% significance level.
We conclude that the growth rate of the boy group is always greater than
that of the girl group between the age 8 and 14.


\subsubsection*{Acknowledgment}

The authors thank Takashi Tsuchiya for his helpful comments on symmetric cone programming.


\begin{thebibliography}{99}

\bibitem[\protect\citeauthoryear{Adler and Taylor}{2007}]{Adler-Taylor07}
Adler, R.\,J.\ and Taylor, J.\,E.\ (2007).
\textit{Random Fields and their Geometry\/}.
Springer, New York.

\bibitem[\protect\citeauthoryear{Barvinok}{2002}]{Barvinok02}
Barvinok, A.\ (2002).
{\it A Course in Convexity\/}.
AMS, Providence, Rhode Island.

\bibitem[\protect\citeauthoryear{Crowder and Hand}{1990}]{Crowder-Hand90}
Crowder, M.\,J.\ and Hand, D.\,J.\ (1990).
\textit{Analysis of Repeated Measures\/}.
Chapman \& Hall/CRC, Boca Raton.

\bibitem[\protect\citeauthoryear{Federer}{1996}]{Federer96}
%
Federer, H.\ (1996).
\textit{Geometric Measure Theory\/}.
Springer, Berlin. 

\bibitem[\protect\citeauthoryear{Johnstone and Siegmund}{1989}]{Johnstone-Siegmund89}
Johnstone, I.\ and Siegmund, D.\ (1989).
On Hotelling's formula for the volume of tubes and Naiman's inequality.
\textit{Ann.\ Statist.\/}, \textbf{18} (1), 652--684.

\bibitem[\protect\citeauthoryear{Karlin and Studden}{1966}]{Karlin-Studden66}
Karlin, S.\ and Studden, W.\ (1966).
\textit{Tchebycheff Systems: With Applications in Analysis and Statistics\/}.
Interscience Publishers, Wiley, New York.

\bibitem[\protect\citeauthoryear{Kato, Yamada and Fujikoshi}{2010}]{Kato-etal10}
Kato, N., Yamada, T., and Fujikoshi, Y.\ (2010).
High-dimensional asymptotic expansion of LR statistic for testing
 intraclass correlation structure and its error bound.
\textit{J.\ Multivariate Anal.\/}, \textbf{101} (1), 101--112. 

\bibitem[\protect\citeauthoryear{Knowles and Siegmund}{1989}]{Knowles-Siegmund89}
Knowles, M.\ and Siegmund, D.\ (1989).
On Hotelling's approach to testing for a nonlinear parameter in regression.
\textit{Internat.\ Statist.\ Rev.\/}, \textbf{57} (3), 205--220.

\bibitem[\protect\citeauthoryear{Kuriki and Takemura}{2000}]{Kuriki-Takemura00}
Kuriki, S.\ and Takemura, A.\ (2000).
Some geometry of the cone of nonnegative definite matrices and weights of
associated chi-bar-square distribution.
\textit{Ann.\ Inst.\ Statist.\ Math.\/}, \textbf{52} (1), 1--14.

\bibitem[\protect\citeauthoryear{Kuriki and Takemura}{2001}]{Kuriki-Takemura01}
Kuriki, S.\ and Takemura, A.\ (2001).
Tail probabilities of the maxima of multilinear forms and their applications.
\textit{Ann.\ Statist.\/}, \textbf{29} (2), 328--371.

\bibitem[\protect\citeauthoryear{Kuriki and Takemura}{2009}]{Kuriki-Takemura09}
Kuriki, S.\ and Takemura, A.\ (2009).
Volume of tubes and the distribution of the maximum of a Gaussian random field. 
\textit{Selected Papers on Probability and Statistics\/},
AMS Translations Series 2,
 \textbf{227}, No.\ 2, 25--48.

\bibitem[\protect\citeauthoryear{Liu}{2010}]{Liu10}
Liu, W.\ (2010). 
Simultaneous Inference in Regression.
CRC Press, Boca Laton, Florida.

\bibitem[\protect\citeauthoryear{Liu, et al.}{2009}]{Liu-etal09}
Liu, W., Bretz, F., Hayter, A.\,J., and Wynn, H.\,P.\ (2009).
Assessing non-superiority, non-inferiority or equivalence when comparing two regression models over a restricted covariate region.
\textit{Biometrics}, \textbf{65} (4), 1279--1287.

\bibitem[\protect\citeauthoryear{Naiman}{1990}]{Naiman90}
Naiman, D.\,Q.\ (1990).
Volumes of tubular neighborhoods of spherical polyhedra and statistical inference.
\textit{Ann.\ Statist.\/}, \textbf{18} (2), 685--716.

\bibitem[\protect\citeauthoryear{Nesterov}{2000}]{Nesterov00}
Nesterov, Y.\ (2000).
Squared functional systems and optimization problems.
In \textit{High Performance Optimization} %
(eds.\ H.\,Frenk, K.\,Roos, T.\,Terlaky and S.\,Zhang), 405--440,
Kluwer, Dordrecht. 

\bibitem[\protect\citeauthoryear{Potthoff and Roy}{1964}]{Potthoff-Roy64}
Potthoff, R.\ and Roy, S.\ (1964).
A generalized multivariate analysis of variance model useful especially for growth curve problems.
\textit{Biometrika\/}, \textbf{51} (3-4), 313--326.

\bibitem[\protect\citeauthoryear{Robertson, et al.}{1988}]{Robertson-etal88}
Robertson, T., Wright, F.\,T.\ and Dykstra, R.\,L.\ (1988).
\textit{Order Restricted Statistical Inference\/},
Wiley, Chichester.

\bibitem[\protect\citeauthoryear{Shapiro}{1988}]{Shapiro88}
Shapiro, A.\ (1988).
Towards a unified theory of inequality constrained testing in multivariate analysis.
\textit{Internat.\ Statist.\ Rev.\/}, \textbf{56} (1), 49--62.

\bibitem[\protect\citeauthoryear{Sturm}{1999}]{Sturm99}
Sturm, J.\,F.\ (1999).
Using SeDuMi 1.02, A Matlab toolbox for optimization over symmetric cones. 
\textit{Optimization Methods and Software\/}, \textbf{11} (1), 625--653.

\bibitem[\protect\citeauthoryear{Sun}{1993}]{Sun93}
Sun, J.\ (1993).
Tail probabilities of the maxima of Gaussian random fields.
\textit{Ann.\ Probab.\/}, \textbf{21} (1), 34--71.

\bibitem[\protect\citeauthoryear{Takemura and Kuriki}{1997}]{Takemura-Kuriki97}
Takemura, A.\ and Kuriki, S.\ (1997).
Weights of chi-bar-square distribution for smooth or piecewise smooth cone alternatives.
\textit{Ann.\ Statist.\/}, \textbf{25} (6), 2368--2387.

\bibitem[\protect\citeauthoryear{Takemura and Kuriki}{2002}]{Takemura-Kuriki02}
Takemura, A.\ and Kuriki, S.\ (2002).
On the equivalence of the tube and Euler characteristic methods for the distribution of the maximum of Gaussian fields over piecewise smooth domains.  
\textit{Ann.\ Appl.\ Probab.\/}, \textbf{12} (2), 768--796.

\bibitem[\protect\citeauthoryear{Uusipaikka}{1983}]{Uusipaikka83}
Uusipaikka, E.\ (1983).
Exact confidence bands for linear regression over intervals.
\textit{J.\ Amer.\ Statist.\ Assoc.\/}, \textbf{78} (383), 638--644.

\bibitem[\protect\citeauthoryear{Working and Hotelling}{1929}]{Working-Hotelling29}
Working, H.\ and Hotelling, H.\ (1929).
Applications of the theory of error to the interpretation of trends.
\textit{J.\ Amer.\ Statist.\ Assoc.\/}, \textbf{26} (165, Supplement), 73--85.

\bibitem[\protect\citeauthoryear{Wynn}{1975}]{Wynn75}
Wynn, H.\,P.\ (1975).
Integrals for one-sided confidence bounds: a general result.
\textit{Biometrika}, \textbf{62} (2), 393--396.

\bibitem[\protect\citeauthoryear{Wynn and Bloomfield}{1971}]{Wynn-Bloomfield71}
Wynn, H.\,P.\ and Bloomfield, P.\ (1971).
Simultaneous confidence bands in regression analysis (with discussions).
\textit{J.\ Roy.\ Statist.\ Soc., Ser.\ B\/}, \textbf{33} (2), 202--221.

\end{thebibliography}
\end{document}